\documentclass[preprint]{elsarticle}

\usepackage{lineno,hyperref}
\usepackage{mathtools}
\usepackage[font=small,labelfont=bf]{caption}
\usepackage{ifthen}

\usepackage{amsthm}
\usepackage{amsmath}
\usepackage{amssymb}
\usepackage{xspace}
\usepackage[T1]{fontenc}
\usepackage{bbold}
\usepackage{cleveref}
\usepackage{lmodern, textcomp}
\DeclareMathAlphabet{\mathsuet} {T1} {wesu}{bx}{sl}

\newcommand{\Fraisse}{Fra\"{i}ss\'{e}\xspace} 
\newcommand{\Jarik}{Ne\v set\v ril\xspace}
\newcommand{\N}{\mathbb{N}} % natural numbers
\newcommand{\age}[1]{\mathrm{Age}(#1)}% Age
\newcommand{\trg}{{\normalfont(}$\bigtriangleup${\normalfont )}\xspace}
\newcommand{\ctrg}{{\normalfont(}$\therefore${\normalfont)}\xspace}
\newcommand{\img}[1]{\mathrm{im}(#1)\xspace}	% Image
\newcommand{\dom}[1]{\mathrm{dom}(#1)\xspace}	% Domain
\newcommand{\kk}[1]{\mathcal K(#1)}		% Substrs. with cone
\newcommand{\okk}[1]{{\overline{\mathcal K}}(#1)}%substrs. no cone

\newcommand{\str}{{\normalfont(}$\ast${\normalfont)}\xspace}
\newcommand{\dgr}{{\normalfont(}$\dagger${\normalfont)}\xspace}

\newcommand{\hh}[1]{\mathcal H(#1)}
\newcommand{\ohh}[1]{{\overline{\mathcal H}}(#1)}

\newtheorem{theorem}{Theorem}
\newtheorem{lemma}[theorem]{Lemma}
\newtheorem{corollary}[theorem]{Corollary}
\newtheorem{proposition}[theorem]{Proposition}
\newtheorem{definition}[theorem]{Definition}

\newdefinition{construction}{Construction}
\newdefinition{remark}{Remark}
\newdefinition{example}{Example}
\newdefinition{notation}{Notation}

\newtheorem{observation}[theorem]{Observation}

\begin{document}
\begin{frontmatter}
\title{The poset of morphism-extension classes of countable graphs}
%\tnotetext[mytitlenote]{Fully documented templates are available in the elsarticle package on \href{http://www.ctan.org/tex-archive/macros/latex/contrib/elsarticle}{CTAN}.}

%% Group authors per affiliation:
\author[ia]{Andr\'{e}s Aranda}
\address[ia]{Institut f\"ur Algebra, Technische Universit\"{a}t Dresden, Zellescher Weg 12-14, Dresden.}
\ead{andres.aranda@gmail.com}%, tdhc@st-andrews.ac.uk.}
%\author[sta]{Thomas D.H. Coleman}
%\address[sta]{School of Mathematics and Statistics, University of St Andrews, St Andrews, KY16 9SS, United Kingdom.}
%\ead{tdhc@st-andrews.ac.uk.}
%\fntext{Since 1880.}
\begin{abstract}
Let $\mathrm{XY_{L,T}}$ denote the class of countably infinite $L$-structures that satisfy the axioms $T$ and in which all homomorphisms of type X (these could be homomorphisms, monomorphisms, or isomorphisms) between finite substructures of $M$ are restrictions of an endomorphism of $M$ of type Y (for example, an automorphism or a surjective endomorphism). Lockett and Truss \cite{LockettTruss:2014} introduced 18 such \emph{morphism-extension classes} for relational structures. For a given pair $L,T$, however, two or more morphism-extension properties may define the same class of structures.

In this paper, we establish all equalities and inequalities between morphism-extension classes of countable (undirected, loopless) graphs.
\end{abstract}


\end{frontmatter}
\section{Introduction}
The notion of homomorphism-homogeneity was introduced by Cameron and \Jarik in \cite{CameronNesetril:2006} as a generalization of ultrahomogeneity in which homomorphisms whose domain is a finite substructure of $M$ (\emph{local homomorphisms}) are restrictions of endomorphisms. Later, Lockett and Truss \cite{LockettTruss:2014} introduced finer distinctions in the class of homo\-morphism-homogeneous $L$-structures, characterized by the type of homomorphism between finite induced substructures of $M$ and the type of endomorphism to which such homomorphisms can be extended. In total, they introduced 18 \emph{morphism-extension classes}, partially ordered by inclusion. 

We call a relational structure $M$ XY-homogeneous if every X-morphism between finite induced substructures extends to a Y-morphism $M\to M$, where $\mathrm{X\in\{I,M,H\}}$ and $\mathrm{Y\in\{H,I,A,E,B,M\}}$. The meaning of these symbols is as follows:
\vspace{-0.2cm}\begin{itemize}
\setlength\itemsep{0em}
\item[$\ast$]{H: homomorphism.}
\item[$\ast$]{M: monomorphism (injective homomorphism).}
\item[$\ast$]{I: isomorphism; an isomorphism $M\to M$ is also called a self-embedding.}
\item[$\ast$]{A: automorphism, (surjective isomorphism $M\to M$).}
\item[$\ast$]{E: epimorphism, (surjective homomorphism).}
\item[$\ast$]{B: bimorphism, (surjective monomorphism).}
\end{itemize}

For example, ultrahomogeneous structures are IA-homogeneous structures in this formulation, and the homomorphism-homogeneous structures of Cameron and \Jarik are our HH-homogeneous structures. The partial order of morphism-extension classes of a general class of countable relational structures is presented in Figure \ref{fig:ctblestrs}.
\begin{figure}[h!]
\centering
\includegraphics[scale=0.7]{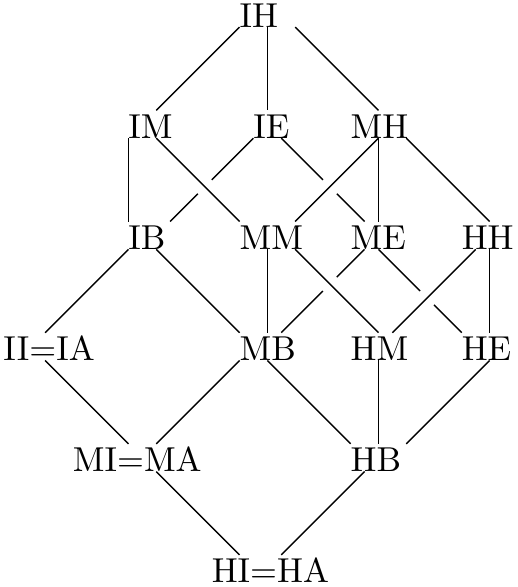}
\caption{Morphism extension classes of countable structures, ordered by $\subseteq$.}
\label{fig:ctblestrs}
\end{figure}

We mention here that further refinements exist; for example, polymorphism-homogeneous structures (a subclass of the class of HH-homogeneous structures) were studied in \cite{pech2015polymorphism}. In this paper we will only consider the 18 classes from \cite{LockettTruss:2014}.

The partial order of morphism-extension classes depends on the type of structures that one considers (graphs, partial orders, directed gaphs, etc.). For each type of relational structure we can ask what its partial order of morphism-extension classes looks like. In this paper we extend results from Rusinov and Schweitzer \cite{RusinovSchweitzer:2010}, who proved the equalities MH=HH for countable graphs and MH=HH=MM for countable connected graphs. We will establish all other equalities between morphism-extension classes of graphs and provide examples for all proper inclusions. In the course of our proofs, some structural information about the graphs in MB, HE, and ME will be derived.

This is part of an effort towards a classification of homomorphism-homogeneous graphs, by which we mean, a collection of lists of structures in each morphis-extension class, up to some suitable equivalence.

%%%%%%%%%%%%%%%%%%%%%%%%%%%%%%%%%%%%%%%%%%%%%%%%%%%%%%%%%%%%%%%%%%%%%%%%%%%%%%%%%%%%%%
%%%%%%%%%%%%%%%%%%%%%%SECTION 2%%%CLASSIFICATIONS,CONVENTIONS,LINGO%%%%%%%%%%%%%%%%%%%
%%%%%%%%%%%%%%%%%%%%%%%%%%%%%%%%%%%%%%%%%%%%%%%%%%%%%%%%%%%%%%%%%%%%%%%%%%%%%%%%%%%%%%
\section{Classification results, conventions, and lingo}
The purpose of this section is to familiarize the reader with some results from the literature that will be useful later, and to introduce some special terms. 

Most of our notation is standard, but for clarity we mention that the edge relation will be denoted by $\sim$, $\overline G$ denotes the complement of $G$, $f_X$ is the restriction of $f$ to $X$, and all subgraphs in the text are induced. The \emph{degree} of a vertex $v\in G$ is the cardinality of $\{w\in G:v\sim w\}$, and its \emph{codegree} in $G$ is its degree in $\overline{G}$. A homomorphism (or monomorphism, or isomorphism) $f\colon A\to B$ where $A$ is a finite subgraph of a countably infinite graph will be called a \emph{local} or \emph{finite} homomorphism (or monomorphism, or isomorphism).

Given two graphs $G$ and $H$ with disjoint vertex sets, we can form the \emph{graph composite} or \emph{lexicographic product} of $G$ and $H$, denoted by $G[H]$, as follows: the vertex set is $G\times H$ and $(g,h)\sim(g',h')$ if $g\sim g'$ in $G$ or $g=g'$ and $h\sim h'$ in $H$. In $G[H]$, each set of the form $\{g\}\times H$ induces an isomorphic copy of $H$ and for any function $f:G\to H$, the set $\{(g,f(g)):g\in G\}$ with its induced subgraph structure in $G[H]$ is isomorphic to $G$. We will use $I_\kappa$ to denote an independent set (empty graph or subgraph) of size $\kappa$.

We remind the reader of some famous graphs. The \emph{Rado graph} $\mathcal R$ is the \Fraisse limit of the class of all finite graphs, and as such it is an ultrahomogeneous (IA) graph. It is also characterised among countably infinite graphs by the following \emph{extension axioms}: for all finite disjoint $A,B\subset\mathcal R$, there exists a vertex $v\in\mathcal R$ such that $v\sim a$ for all $a\in A$, and $v\not\sim b$ for all $b\in B$. The \emph{universal homogeneous $K_n$-free graph} $\mathcal H_n$ is the \Fraisse limit of the class of all finite $K_n$-free graphs, and is characterised among countable graphs by similar extension axioms, with the additional requirement that $A$ be $K_{n-1}$-free. 

Classification results are particularly useful to separate the classes. In the case of homomorphism-homogeneous graphs, two infinite classes have been classified.

\begin{theorem}[Lachlan-Woodrow 1980 \cite{LachlanWoodrow:1980}]\label{thm:lw}
Let $G$ be a countably infinite IA-homogeneous graph. Then $G$ or its complement is isomorphic to one of the following:
\begin{enumerate}
\item{$K_\omega$,}
\item{$I_\omega[K_n]$ for some $n\in\omega, n\geq 2$,}
\item{$I_n[K_\omega]$ for some $n\in\omega+1, n\geq2$,}
\item{the Rado graph $\mathcal R$, or }
\item{the universal homogeneous $K_n$-free graph, for some $n\in\omega, n\geq 3$.}
\end{enumerate}
\end{theorem}

\begin{theorem}[Aranda-Hartman 2020 \cite{aranda2019independence}]\label{thm:ah}
Let $G$ be a countably infinite MB-homogeneous graph. Then $G$ or its complement is bi\-morphism-equivalent to one of the following:
\begin{enumerate}
\item{$K_\omega$,}
\item{$I_\omega[K_\omega]$,}
\item{The Rado graph $\mathcal R$.}
\end{enumerate}
\end{theorem}

Bimorphism-equivalence is the relation that holds between two structures $M$ and $N$ if there exist bijective homomorphisms $M\to N$ and $N\to M$. 

In this paper, we will use the symbol XY to denote the class of countable XY-homogeneous graphs, up to isomorphism. This means, in particular, than an equation like HM=$\{K_\omega\}$ should be interpreted as saying that every countable HM-homogeneous graph is isomorphic to $K_\omega$.

 Since we know all the graphs in MB and IA, the following result tells us that there are no new graphs in IB.

\begin{theorem}[Aranda 2019 \cite{aranda2019ib}]\label{thm:aa}
For countable graphs, $\mathrm{IB}=\mathrm{MB}\cup\mathrm{IA}$.
\end{theorem}

Some equalities between morphism-extension classes of graphs are known. 
\begin{theorem}[Rusinov-Schweitzer 2010]\label{thm:rs}
A countable graph $G$ is MH-homo\-geneous iff it is HH-homogeneous. Moreover, the only countable MH=HH-homo\-geneous graphs that are not MM-homogeneous are of the form $I_\omega[K_n]$ with finite $n\geq 2$.
\end{theorem}

These two facts from Cameron-Ne\v set\v ril and Rusinov-Schweitzer will be used later.

\begin{theorem}\label{thm:2facts}~
\begin{enumerate}
\item{If $G$ is an infinite connected HH-homogeneous graph, then $G$ does not contain finite maximal cliques; in particular, every vertex has infinite degree.}
\item{If $G$ is an infinite connected HH-homogeneous graph, $H\subset G$ is finite and $c\in G$ is a cone over $H$, then there are infinitely many cones over $H$ in $G$.}\label{fact2}
\end{enumerate}
\end{theorem}

Finally, we paraphrase a result about disconnected HH-homogeneous graphs:

\begin{theorem}[Cameron-\Jarik 2006]\label{thm:iasonhh}
All disconnected IH-homogeneous graphs are HH-homogeneous.
\end{theorem}

The following two properties will appear in several sections in this work. We can trace the first one to the original Cameron-\Jarik paper, while the second one comes from \cite{ColemanEvansGray:2019}. One may think of them as the positive (the edge relation appears without negations) and negative parts of the extension axioms for the Rado graph.

\begin{definition}
A graph $G$ has property \trg if for every finite $H\subset G$ there exists $v\in G$ with $v\sim h$ for all $h\in H$. Such $v$ will be called a \emph{cone} over $H$.
\end{definition}

\begin{definition}
A graph $G$ has property \ctrg if for every finite $H\subset G$ there exists $v\in G\setminus H$ such that for all $h\in H$ the pair $\{v,h\}$ is a nonedge in $G$. Equivalently, $G$ satisfies \ctrg iff $\overline G$ satisfies \trg. A vertex $v\notin X$ such that for all $x\in X$ we have $v\not\sim x$ will be called a \emph{co-cone} over $X$.
\end{definition}

The relevance of these two properties is that \trg was observed to imply HH-homogeneity in \cite{CameronNesetril:2006}, and it can be used to find images for a given vertex when one wishes to extend the domain of a local homomorphism. Similarly, \ctrg can be used to find preimages when one is trying to extend the image of a homomorphism, and is associated with HE-homogeneity. Later in this paper, we will use weakened versions of these two properties to establish ME=HE for general countable graphs and $ME=HE=MB$ for connected countable graphs.

The following proposition is an amalgamation of results from \cite{ColemanEvansGray:2019} and \cite{CameronNesetril:2006} linking the properties above to the Rado graph.
\begin{proposition}
Let $G$ be a countable graph. Then 
\begin{enumerate}
\item{$G$ satisfies \trg iff there exists a bijective homomorphism $\mathcal R\to G$, and}
\item{$G$ satisfies \trg and \ctrg\ iff it is bimorphism equivalent to $\mathcal R$.}
\end{enumerate}
\end{proposition}

The first examples of HH-homogeneous graphs that do not satisfy \trg were found by Rusinov and Schweitzer (\cite{RusinovSchweitzer:2010}). They can be described as the graphs $\mathrm{RS}(n)$ ($n\geq2$) with vertex set $\N$ and edge set $$\{\{k,t\}:k\geq n\wedge t\geq n\wedge k\neq t\}\cup\{\{k,t\}:k\geq n\wedge t\leq n-1\wedge k\not\equiv t\mod{n}\}.$$ We refer to these graphs as \emph{the Rusinov-Schweitzer examples.}

Throughout the paper, we will make implicit use of the following observations. 

\begin{observation}
A countable structure $M$ is XY-homogeneous if and only if for all finite \emph{surjective} X-morphisms $f\colon A\to B$ there exists an endomorphism of type $Y$ such that $F_A=f$.
\end{observation}

\begin{observation}\label{obs:hn}
The universal homogeneous $K_n$-free graph ($n\geq3$) belongs to all classes of the form IY, but not to any other morphism-extension class.
\end{observation}
\begin{proof}
The first statement follows from the facth that $\mathcal{H}_n$ is IA-homogeneous, and IA is contained in all other IY-homogeneity classes.

To see that $\mathcal{H}_n$ is not XY-homogeneous if X$\neq$I, it suffices to see that it is not MY-homogeneous for any Y. By the extension axioms of $\mathcal{H}_n$, each nonedge of $\mathcal{H}_n$ is contained in a copy of $K_n^-$, the graph on $n$ vertices with ${n}\choose{2}-1$ edges. Therefore, the monomorphism mapping a nonedge to an edge is not restriction of any endomorphism of $\mathcal{H}_n$.
\end{proof}

\begin{theorem}
The following inclusions are proper in the class of countably infinite graphs.
\begin{enumerate}
\item{$\mathrm{IA=II\subset IB}$,}
\item{$\mathrm{MB\subset IB}$,}
\item{$\mathrm{MB\subset MM}$.}
\end{enumerate}
\end{theorem}
\begin{proof}
The first two inequalities follow from Theorems \ref{thm:lw} and \ref{thm:ah}.

Any graph that is bimorphism-equivalent to $\mathcal R$, but not isomorphic to it, is IB-homogeneous but not II=IA-homogeneous.

The universal homogeneous $K_n$-free graph ($n\geq 3$) is IB-homogeneous, but not MB-homogeneous. See section 3.1 of \cite{ColemanEvansGray:2019} for a construction that yields uncountably many pairwise non-isomorphic such graphs.

The Rusinov-Schweitzer examples are MM-homogeneous by Theorem \ref{thm:rs}, but not MB-homogeneous by Theorem \ref{thm:ah}.
\end{proof}

%%%%%%%%%%%%%%%%%%%%%%%%%%%%%%%%%%%%%%%%%%%%%
%%%%SECTION: BOTTOM%%%%%%%%%%%%%%%%%%%%%%%%%%
%%%%%%%%%%%%%%%%%%%%%%%%%%%%%%%%%%%%%%%%%%%%%

\section{The bottom of the hierarchy}\label{sec:bottom}

We start by proving that the disconnected elements of IH are disjoint unions of cliques. The argument is the same as in \cite{CameronNesetril:2006}, but with relaxed hypotheses.

A connected component of a graph is \emph{nontrivial} if it contains at least two vertices.

\begin{lemma}\label{lem:disccliques}
Let $G$ be a countably infinite disconnected IH-homogeneous graph. Then there exist $n,m\in\omega+1$ with $\max\{m,n\}=\omega$ such that $G\cong I_n[K_m]$.
\end{lemma}
\begin{proof}
The result is obvious if all the connected components of $G$ are trivial, as in that case the fact that $G$ is countable yields immediately $G\cong I_\omega\cong I_\omega[K_1]$. 

Suppose then that $C$ and $D$ are distinct connected components of $G$ and $C$ is nontrivial. 

Our first claim is that every connected component is nontrivial. Since $C$ is nontrivial, there exist $c_1,c_2\in C$ with $c_1\sim c_2$. Now let $v$ be any vertex in $G$, and consider the local isomorphism $c_1\mapsto v$. By IH-homogeneity, this map is restriction of an endomorphism $F$ of $G$, and it follows that $F(c_2)$ is a neighbour of $v$.

Now we claim that each connected component of $G$ is a clique. Since $G$ has a nontrivial component, all components are nontrivial. Suppose for a contradiction that $D$ is not a clique, so there exist $d_1,d_2,d_3\in D$ with $d_1\not\sim d_2$ and $d_3\sim d_1,d_2$. The local isomorphism $d_1\mapsto d_1, d_2\mapsto c_1$ is not restriction of any endomorphism of $G$ because $c_1$ and $d_1$ are in different connected components.

Finally, to establish that all connected components are of the same size, consider the local isomorphisms $f_1\colon c_1\mapsto d_1$ and $f_2\colon d_1\mapsto c_1$. An endomorphism extending $f_1$ restricts to an injection $C\to D$, and likewise an endomorphism extending $f_2$ restricts to an injection $D\to C$, and the result follows from the Cantor-Bernstein theorem.
\end{proof}

\begin{theorem}\label{thm:bottom}
In the class of countably infinite graphs,
\begin{enumerate}
\item{HM=HI=HB=HA=$\{K_\omega\}$.}\label{it:hmhi}
\item{MI=MA=$\{K_\omega,I_\omega\}$.}\label{it:mima}
\end{enumerate}
\end{theorem}
\begin{proof}
It is easy to verify that $K_\omega$ is in each of these classes. HI=HA and MI=MA are true for general countable structures (Lemma 1.1. in \cite{LockettTruss:2014}).

The equalities in the first item are proved as follows: if every homomorphism between finite substructures is restriction of an injective endomorphism of $G$, then it must be the case that all local homomorphisms are injective. From this it follows that $G$ is complete, and therefore isomorphic to $K_\omega$.

To prove $\mathrm{MA}=\{K_\omega,I_\omega\}$, note that an MA-homogeneous graph cannot contain both edges and nonedges, because the monomorphism mapping a nonedge to an edge is never restriction of an automorphism. It follows that an MA-homogeneous graph is complete or empty; the fact that the empty countable graph is MA-homogeneous is easy to verify.
\end{proof}

\begin{corollary}\label{cor:discineq}
The following inclusions are proper in the class of countably infinite graphs:
\begin{enumerate}
\item{$\mathrm{HI\subset MI}$,}
\item{$\mathrm{MI\subset II}$,}
\item{$\mathrm{MI\subset MB}$,}
%\item{$\mathrm{MB\subset HE}$,}
\item{$\mathrm{MM\subset MH}$,}
\end{enumerate}
\end{corollary}
\begin{proof}~
\begin{enumerate}
\item{The infinite independent set is MI-homogeneous but not HI-homogeneous.}
\item{The universal homogeneous $K_n$-free graph is II-homogeneous but not MI-homogeneous (Observation \ref{obs:hn}).}
\item{$I_\omega[K_\omega]$ is MB-homogeneous (Theorem 4.11 in \cite{ColemanEvansGray:2019}) but not MI-homogeneous.}
\item{$I_\omega[K_n]$ is HH-homogeneous (Theorem \ref{thm:iasonhh}) but not MM-homogeneous (there are no injective endomorphisms extending the monomorphism that maps a nonedge to an edge).}
\end{enumerate}
\end{proof}

%%%%%%%%%%%%%%%%%%%%%%%%%%%%%%%%%%%%%%%%%%%%%%%%%%%%%%%%%%%%%
%%%%%%%%%SECTION: PROPERTIES DGR AND STR%%%%%%%%%%%%%%%%%%%%%
%%%%%%%%%%%%%%%%%%%%%%%%%%%%%%%%%%%%%%%%%%%%%%%%%%%%%%%%%%%%%
\section{Asterisk \& Obelisk in Homogenia}
In this subsection, we introduce two properties which we will use to establish ME=HE for general countable graphs and ME=HE=MB for countable connected graphs.

For graphs, \str and \dgr are  weaker than \trg and \ctrg, respectively. They are explicitly stated in terms of local homomorphisms (without reference to the language of graphs), so we define them for general structures.

\begin{definition}\label{def:str}
A structure $G$ satisfies \str if for every surjective finite monomorphism $f:A\to B$ and every $c\notin A$ there exists $d\notin B$ such that $f\cup\{(c,d)\}$ is a homomorphism.
\end{definition}

\begin{remark}\label{rmk:strmh}
Since we require $d\notin B$, Property \str guarantees the possibility of extending any local monomorphism as a monomorphism, and is equivalent to MM-homogeneity for countable structures.
\end{remark}

\begin{definition}
A structure $G$ satisfies Property \dgr if for all finite surjective homomorphisms $f:A\to B$ and $b\notin B$, there exists $a\notin A$ such that $f\cup\{(a,b)\}$ is a homomorphism.
\end{definition}

While Property \str is equivalent to MM-homogeneity under the hypothesis of denumerability, Property \dgr only implies that every local homomorphism is restriction of a surjective partial endomorphism.

\begin{proposition}\label{prop:HHctrg}
If $G$ is a countably infinite HH-homogeneous graph with property \dgr, then $G$ is HE-homogeneous.
\end{proposition}
\begin{proof}
Let $f:H\to H'$ be a surjective homomorphism between finite induced substructures of $G$. Enumerate $G$ in two different ways, $G=\{a_i:i\in\omega\}$, $G=\{c_i:i\in\omega\}$, so that $\{a_i:i\leq n\}=H$, $\{c_i:i\leq m\}=H'$. We will extend $f$ by a back-and-forth argument. 

Even steps: Suppose that $f_{2k}$ is a homomorphism extending $f$ (this is trivially true for $f_0=f$). Let $e$ be the least index for which $c_{e}$ is not in $\img{f_k}$. Since $G$ satisfies \dgr, there exists some $a_{s}\notin\dom{f_{2k}}$ such that $f_{2k+1}:=f_k\cup\{(a_{s},c_{e})\}$ is a homomorphism.

Odd steps: Suppose that $f_{2k+1}$ is a homomorphism that extends $f$. Let $q$ be the least index for which $f_{2k+1}$ is not defined on $a_{q}$. Since $G$ is HH-homogeneous, there exists some $c_{r}$ such that $f_{2(k+1)}=f_{2k+1}\cup\{(a_{q},c_{r})\}$ is a homomorphism.

Now we claim that $\hat f=\bigcup\{f_k:k\in\N\}$ is a surjective endomorphism of $G$. By construction, its image is all of $G$, and moreover any restriction to a finite domain is restriction of some $f_m$, which is a local homomorphism; from this it follows that $G$ preserves edges.
\end{proof}

We mention two special cases of Proposition \ref{prop:HHctrg} for ease of reference.
\begin{corollary}\label{cor:dgrHE}~
\begin{enumerate}
\item{If $G$ is a countably infinite ME-homogeneous graph with \dgr, then $G$ is HE-homogeneous.}\label{dgrhe1}
\item{If $G$ is a countably infinite MB-homogeneous graph with \dgr, then $G$ is HE-homogeneous.}\label{dgrhe2}
\end{enumerate}
\end{corollary}
\begin{proof}
ME- and MB-homogeneous graphs are MH-homogeneous, and by Theorem \ref{thm:rs}, HH-homogeneous. The result now follows from Proposition \ref{prop:HHctrg}
\end{proof}

\begin{proposition}\label{prop:NoStr}
If $G$ is an HE-homogeneous graph with property \str, then $G$ is MB-homogeneous.
\end{proposition}
\begin{proof}
Let $f:H\to H'$ be a surjective monomorphism between finite induced substructures of $G$. Enumerate $G$ in two different ways, $G=\{a_i:i\in\omega\}$, $G=\{c_i:i\in\omega\}$, so that $\{a_i:i\leq n\}=H$, $\{c_i:i\leq n\}=H'$, and $f$ maps $a_i\mapsto c_i$ for $i<n$. We will extend $f$ by a back-and-forth argument in which the even steps use HE-homogeneity to ensure surjectivity and the odd steps ensure injectivity and that every element appears eventually in the domain of some extension. 

Even steps: Suppose that $f_{2k}$ is a monomorphism extending $f$ (this is trivially true for $f_0=f$). Let $e$ be the least index for which $c_{e}\notin\img{f_k}$. Since $G$ is HE, there exists some $a_{s}$ such that $f_{2k+1}:=f_k\cup\{(a_{s},c_{e})\}$ is a homomorphism. In particular, it is a function and $a_s\notin\dom{f_{2k}}$, so $f_{2k+1}$ is a monomorphism covering $c_{e}$. 

Odd steps: Suppose that $f_{2k+1}$ is a monomorphism that extends $f$. Let $q$ be the least index for which $f_{2k+1}$ is not defined on $a_{q}$. Since $G$ has \str, there exists some $c_{r}$ such that the extension $f_{2(k+1)}=f_{2k+1}\cup\{(a_{q},c_{r})\}$ is a monomorphism defined on $a_{q}$.

We claim that $\hat f=\bigcup\{f_k:k\in\N\}$ is a bimorphism of $G$. It is a bijection $G\to G$ by construction, and any restriction to a finite domain is a monomorphism, so $\hat f$ is an endomorphism.
\end{proof}

The proof of the following proposition follows the same pattern as those of Propositions \ref{prop:HHctrg} and \ref{prop:NoStr}, and is left to the reader.
\begin{proposition}\label{prop:ihdgr}
If $G$ is IH-homogeneous and satisfies \dgr, then $G$ is IE-homogeneous.
\end{proposition}

%%%%%%%%%%%%%%%%%%%%%%%%%%%%%%%%%%%%%%%%%%%%%%%%%%%%%%%%%%%%%
%%%%%SECTION:CONDITIONS HH HE ME%%%%%%%%%%%%%%%%%%%%%%%%%%%%%
%%%%%%%%%%%%%%%%%%%%%%%%%%%%%%%%%%%%%%%%%%%%%%%%%%%%%%%%%%%%%
\section{Conditions on $\age{G}$ characterising HH, HE, ME}
In this subsection, we give properties of $\age{G}$ that characterise HH-, HE-, and ME-homogeneity for graphs in terms of cones and co-cones. It is often easier to verify these conditions than to prove homogeneity directly.

\begin{definition}
Let $G$ be a countable graph. Define:
\begin{enumerate}
\item{Define $\kk{G}$ as the subset of $\age{G}$ consisting of all $A$ for which there exists an embedding $e:A\to G$ such that $G$ contains a cone over $e[A]$}
\item{Define $\okk G$ as the subset of $\age{G}$ consisting of all $A\in\age{G}$ for which there exists an embedding $e:A\to G$ such that no vertex in $G\setminus e[A]$ is a cone in $G$ over $e[A]$}
\item{$\hh{G}$ as the subset of $\age{G}$ consisting of all $A\in\age{G}$ for which there exists an embedding $e:A\to G$ such that $G$ contains a co-cone over $e[A]$, and}
\item{$\ohh{G}$ as the subset of $\age{G}$ consisting of all $A\in\age{G}$ for which there exists an embedding $e:A\to G$ such that no vertex in $G\setminus e[A]$ is a co-cone over $C$.}
\end{enumerate}
\end{definition}

Each of the pairs $\kk{G}, \okk{G}$ and $\hh{G},\ohh{G}$ covers $\age{G}$, but they do not form partitions in a general graph. 	For example, if $G$ contains a non-trivial connected component and an isolated vertex, then the graph consisting of a single vertex appears in $\kk{G}$ and $\okk{G}$.

\begin{observation}\label{obs:kkokk}
Let $G$ be an IH-homogeneous graph, and suppose that $C\in\kk{G}$. Then for all embeddings $e\colon C\to G$, there is a cone over $e[C]$ in $G$. Similarly, if $C\in\okk{G}$, then no copy of $C$ in $G$ has a cone in $G$.
\end{observation}
\begin{proof}
Suppose for a contradiction that there exist embeddings $e_0,e_1\colon C\to G$ such that there exists a cone $c$ over $e_1[C]$, but not over $e_0[C]$. Then no isomorphism $i\colon e_1[C]\to e_0[C]$ is restriction of an endomorphism, because $c$ cannot be mapped to any vertex. 

A similar argument proves the second statement.
\end{proof}

The proof of Observation \ref{obs:kkokk} can be adapted to prove the following:
\begin{observation}
Let $G$ be an IE-homogeneous graph, and suppose that $C\in\hh{G}$. Then for all embeddings $e\colon C\to G$, there is a co-cone over $e[C]$ in $G$. Similarly, if $C\in\ohh{G}$, then no copy of $C$ in $G$ has a co-cone in $G$.
\end{observation}
\vspace{-0.6cm}\begin{flushright}$\square$\end{flushright}

The two observations above allow us to abuse notation and write $X\in\kk{G}$ (and similar expressions) for a subgraph $X\subset G$ (as opposed to an element of the age).

In a partial order $(P,\leq)$, we call $X\subset P$ \emph{downward closed} if given any $x\in X$ and $y\in P$, $y\leq x$ implies $y\in X$. Similarly, $Y\subset P$ is \emph{upward closed} if for all $y\in Y$ and all $x\in P$, $x\geq y$ implies $x\in Y$. Naturally, we consider the empty set to be upward and downward closed.

Write $A\preceq B$ if there exists a surjective homomorphism $A\to B$. This relation is a partial order on $\age{G}$. Observe that in this order $A\preceq B$ implies $|A|\geq|B|$.

The following result is Proposition 5 from \cite{aranda2019independence}.

\begin{proposition}\label{prop:Conditions}
Let $G$ be a countable graph. Then $G$ is HH-homogeneous iff the following two conditions hold.
\begin{enumerate}
\item{$\kk{G}\cap\okk{G}=\varnothing$ and}\label{DisjCond}
\item{$\kk{G}$ upward-closed in $(\age{G},\preceq)$ (equivalently, $\okk{G}$ downward-closed in $(\age{G},\preceq)$
)}\label{kkCond}\label{okkCond}
\end{enumerate}
\end{proposition}

We are now looking for analogues of Propositions \ref{prop:Conditions} for HE- and ME-homogeneous graphs. We will need a couple of propositions first.

\begin{proposition}\label{prop:infdegcodegr}
If $G$ is a connected HE-homogeneous graph embedding a nonedge, then there are no finite maximal independent sets or cliques in $G$. In particular, each vertex has infinite degree and codegree. 
\end{proposition}
\begin{proof}
HE-homogeneous graphs are HH-homogeneous by Theorem \ref{thm:rs}, and we know that every vertex in a connected HH-homogeneous graphs has infinite degree (Theorem \ref{thm:2facts}). 

We will now prove that $G$ contains no finite maximal independent sets by an induction induction argument with two steps at each stage: first we prove that each vertex has codegree at least $n$, then we prove that each vertex belongs to an independent set of size at least $n+1$.

First observe that every vertex belongs of a nonedge: indeed, if $uv$ is a nonedge, and $w\in G$, then $w\mapsto u$ is a monomorphism and since $\{u\}$ has a co-cone, $w$ must also have one (otherwise, there would be no preimage for $v$ under any global extension of the finite mapping). Therefore, every vertex has co-degree at least 1 (and in particular belongs to an independent set of size 2). This is our basis for induction.

Now suppose that each vertex belongs to an independent set of size $n$. Let $v_0,\ldots,v_{n-1}$ be an independent set and $x\sim v_{n-1}$. Define $f:\{v_0,\ldots,v_{n-1}\}\to\{x,v_1,\ldots,v_{n-1}\}$ by $v_i\mapsto v_{i+1}$ for $i<n-1$ and $v_{n-1}\mapsto x$. This monomorphism can be extended to an epimorphism, so $v_0$ has a preimage $z\notin\{v_1,\ldots,v_n\}$, and in particular $z\not\sim v_0$. It follows that $v_0$ has codegree at least $n$.

At this point, we may assume that each vertex belongs to an independent set of size $n$ and has codegree at least $n$. Let $u_0,\ldots, u_{n-1}$ be an independent set and $v\not\sim u_0$, where we assume $v\notin\{u_0,\ldots,u_{n-1}\}$. Let $g$ be the function fixing $\{u_1,\ldots,u_{n-1}\}$ pointwise and mapping $u_0\mapsto v$. This is a monomorphism and any preimage of $u_0$ under a global extension of $g$ is a co-cone over $u_0,\ldots,u_{n-1}$, so $u_0$ belongs to an independent set of size $n+1$.
\end{proof}

\begin{proposition}\label{prop:CoconesME}
Let $G$ be a connected HE-homogeneous graph. If $C\in\kk{G}$, then there exists an infinite clique of cones over $C$. Similarly, if $D\in\hh{G}$ then there exists an infinite independent set of co-cones over $D$.
\end{proposition}
\begin{proof}
The first statement follows from HH-homogeneity (Theorem \ref{thm:2facts}). To prove the second one, suppose that $d$ is a co-cone over $D$. We know from Proposition \ref{prop:infdegcodegr} that $d$ has infinite codegree, so there is some $w\not\sim u$ not in $D$. The map $f$ fixing $D$ pointwise and sending $d\mapsto w$ is a monomorphism, and any vertex in the preimage of $d$ under an epimorphism extending $f$ is a co-cone over $D\cup\{d\}$. Repeating this argument we can show that there is no finite bound on the size of independent set of co-cones over $D$.
\end{proof}

\begin{remark}\label{rmk:switcheroo}
In the proofs of Propositions \ref{prop:infdegcodegr} and \ref{prop:CoconesME}, the local homomorphisms used were monomorphisms, and we only used surjectivity from the condition of HE-homogeneity. It follows that if we substitute HE in the hypotheses by ME, MA, or MB, then the conclusions are still valid.
\end{remark}

\begin{proposition}\label{prop:ConditionsHE}
A countable graph is HE-homogeneous iff it is HH-homogeneous and the following two conditions on sets of co-cones hold:
\begin{enumerate}
\item{$\hh{G}\cap\ohh{G}=\varnothing$,}\label{Cond5}
\item{$\hh{G}$ is downward-closed (equivalently, $\ohh{G}$ is upward-closed) in \\$(\age{G},\preceq)$.}\label{Cond6a}
\end{enumerate}
\end{proposition}
\begin{proof}
Let $G$ be an HE-homogeneous graph (it is clearly HH-homogeneous as well). Suppose for a contradiction that $C\in\hh{G}\cap\ohh{G}$. Then there exist embeddings $e_0,e_1\colon C\to G$ such that there exists a co-cone $c$ over $e_1[C]$, but not over $e_0[C]$. Then no isomorphism $i\colon e_1[C]\to e_0[C]$ is restriction of an endomorphism, because no vertex can be an image of $c$. Therefore, \ref{Cond5} holds.

Now suppose that $B\in C$ is an element of $\hh{G}$, and let $e\colon B\to G$ be an embedding such that $G$ contains a co-cone over $e[B]$. Suppose that $f\colon A\to B$ is a surjective local homomorphism; by HE-homogeneity, there exists a surjective endomorphism $F$ with $F_A=f$. By surjectivity of $F$, there exists $c\notin A$ with $F(c)=b$. Clearly, $c$ is a co-cone over $A$, and the isomorphism type of $A$ is in $\hh{G}$. 

To prove the converse, suppose that $G$ is a HH-homogeneous graph that satisfies the two conditions from the statement. Let $f\colon A\to B$ be a surjective local homomorphism. 

Given $a\notin A$, HH-homogeneity implies that we can find some $b$ with $f\cup\{(a,b)\}$ is a homomorphism.

Given $b\notin B$, we may assume that $b$ is not a cone over $B$ because in that case any $d\notin A$ can be mapped to $b$ as a homomorphism extending $f$. There are two cases to consider:
\begin{enumerate}
\item{If $A\in\hh{G}$, then for any co-cone $c$ of $A$, the function $f\cup\{(c,b)\}$ is a homomorphism.}
\item{If $A\in\ohh{G}$, then, since $b$ is not a cone, there is a nonempty $X\subset B$ such that $b$ is a co-cone over $X$. Since $\hh{G}$ is downward-closed, the preimage of $X$ under $f$ has a co-cone $d$, which we can choose outside of $A$ by Proposition \ref{prop:CoconesME}. Now $f\cup\{(d,b)\}$ is a homomorphism.}
\end{enumerate}
The equivalence in Condition \ref{Cond6a} follows from the fact that $\hh{G}$ and $\ohh{G}$ form a partition of $\age{G}$.
\end{proof}

The analogue of Proposition \ref{prop:Conditions} for ME-homogeneous graphs requires a small adjustment: instead of the partial order $\preceq$, we will compare structures using the partial order $A\sqsubseteq B$ that holds when there exists a surjective \emph{mono}morphism $A\to B$.
\begin{proposition}\label{prop:ConditionsME}
A countable graph is ME-homogeneous iff it is HH-homogeneous and the following two conditions on sets of co-cones hold:
\begin{enumerate}
\item{$\hh{G}\cap\ohh{G}=\varnothing$,}\label{Cond5}
\item{$\hh{G}$ is downward-closed (equivalently, $\ohh{G}$ is upward-closed) in\\ $(\age{G},\sqsubseteq)$.}\label{Cond6a}
\end{enumerate}
\end{proposition}
\begin{proof}
Suppose thath $G$ is ME-homogeneous. In particular, it is MH-homogeneous, and, by Theorem \ref{thm:rs}, HH-homogeneous. The first condition is satisfied by the same argument as in Proposition \ref{prop:ConditionsHE}. Condition \ref{Cond6a} follows by the same arguments as in Proposition \ref{prop:ConditionsHE}, but with monomorphisms.

Now suppose that $G$ is HH-homogeneous and satisfies conditions \ref{Cond5} and \ref{Cond6a}, and let $f:A\to B$ be a surjective monomorphism. Observe that condition \ref{Cond5} needs to hold, for otherwise we could map a substructure without a co-cone to one with a co-cone, and any global extension of such a mapping could not have the co-cone in its image.

For any $c\notin A$, we can find a $c'$ so that $f\cup\{(c,c')\}$ is a homomorphism by HH-homogeneity, so we need only ensure that $f$ can be extended surjectively. Take a vertex $d'\notin B$; we need to find some $d\notin A$ such that $f\cup\{(d,d')\}$ is a homomorphism. 
\begin{enumerate}
\item{If $d'$ is a co-cone over $B$, then condition \ref{Cond6a} ensures that we can find a suitable $d\notin B$.}
\item{Otherwise, $d'$ is not a co-cone and we can assume that it is also not a cone over $B$, so it is a co-cone over a nonempty $X\subset B$. Restricting $f$ to the preimage of $X$ yields a monomorphism, and applying Proposition \ref{prop:infdegcodegr} again, we can find a preimage for $d'$.}
\end{enumerate}
This completes the proof.
\end{proof}

%%%%%%%%%%%%%%%%%%%%%%%%%%%%%%%%%%%%%%%%%%%%%%%%%%%%%%%%%%%%%
%%%%%%%%%SECTION: MB SUBSET HE%%%%%%%%%%%%%%%%%%%%%%%%%%%%%%%
%%%%%%%%%%%%%%%%%%%%%%%%%%%%%%%%%%%%%%%%%%%%%%%%%%%%%%%%%%%%%
\section{$\mathrm{MB\subseteq HE}$}

\iffalse
An important fact from Coleman-Evans-Gray \cite{ColemanEvansGray:2019}:
\begin{theorem}\label{thm:MBcomp}
If $G$ is a MB-homogeneous graph, then its complement $\overline G$ is also MB-homogeneous.
\end{theorem}
\fi

The following proposition is a consequence of the fact that the class of MB-homogeneous graphs is closed under complements (proved in \cite{ColemanEvansGray:2019}). A proof appears in \cite{aranda2019independence}, where the result is called Corollary 25.
\begin{proposition}\label{cor:InfDegCodeg}
If a countable graph $G$ is MB-homogeneous and neither complete nor empty, then it is connected or isomorphic to $I_\omega[K_\omega]$. Moreover, every vertex has infinite degree and co-degree.
\end{proposition}

\begin{remark}\label{rmk:NegDstr}
By Corollary \ref{cor:dgrHE}, item \ref{dgrhe2}, if $G$ is MB- but not HE-homo\-geneous, then there exist finite substructures $A,B\subset G$ with a surjective homomorphism $f:A\to B$ and a vertex $c\notin B$ such that every function $g$ extending $f$ with $c\in\img{g}$ fails to be a homomorphism.
\end{remark}

We will now give a name to the indicator function of the neighbourhood of a vertex $v$ in a finite set $F$. 
\begin{notation}
In an ambient graph $G$, if $F$ is a finite subset of $G$ and $v\notin F$, we will use $\psi_{v,F}$ to denote the function $$\psi_{v,F}(x)=\begin{cases} 1& \text{if } v\sim x\\ 0& \text{otherwise.}\end{cases}$$
\end{notation} 

\begin{lemma}\label{lem:MBalg}
If $G$ is an MB- but not HE-homogeneous graph, then there exists a finite nonempty $H\subset G$ such that the set $X$ of co-cones over $H$ is finite and nonempty. In particular, $H\cup X$ does not have a co-cone in $G$.
\end{lemma}
\begin{proof}
Following Remark \ref{rmk:NegDstr}, suppose that for some $f:A\to B$ and $c\notin B$, every function $f'$ extending $f$ with $c\in\img{f'}$ fails to be a homomorphism. In particular, for every $e\notin A$, the function $f\cup\{(e,c)\}$ is not a homomorphism. Consider the type function $\psi_{c,B}$; we may assume that it is not constant 1, as any vertex outside of $A$ can be mapped to a cone over $A$ without breaking the conditions for homomorphism. 

Any vertex not related to any element of $Y\coloneqq f^{-1}[\psi_{c,B}^{-1}[0]]$ can be mapped to $c$, and said extension is a homomorphism. It follows that the formula $$\varphi(x,Y)\coloneqq \bigwedge_{e\in Y}e\not\sim x$$ does not have any solutions outside $A$.

Now we claim that $\varphi(x,Y)$ has a solution in $A$. Let $Z_0,\ldots, Z_{k-1}$ be the $\ker{f}$-classes contained in $Y$, and let $t=\{z_0,\ldots, z_{k-1}\}$ be a transversal of $\ker{f_Y}$. Then $f_t$ is a monomorphism, which by MB-homogeneity is restriction of a bimorphism $F\colon G\to G$. In particular, there exists $d\notin t$ with $F(d)=c$. We have already established that no such $d$ exists in $G\setminus A$; it follows that $d$ is an element of $A$.
\end{proof}

By Remark \ref{rmk:switcheroo}, Propositions \ref{prop:infdegcodegr} and \ref{prop:CoconesME} can be modified so that their conclusions apply to MB-homogeneous graphs. We spell out one of the resulting facts below.

\iffalse
\begin{lemma}\label{lem:CoRS}
If $G$ is an MB-homogeneous graph embedding a nonedge, then every nonedge is contained in an infinite independent set.
\end{lemma}
\begin{proof}
We will show that there are no finite maximal independent sets of size $n$, for each $n\in\N$. Since every vertex has infinite co-degree (Proposition \ref{cor:InfDegCodeg}), the result holds for $n=2$. 

Suppose for a contradiction that $\{v_1,\ldots,v_n\}$ is a maximal independent set; $v_1$ has infinite co-degree, so there exists $c\notin\{v_1,\ldots,v_n\}$ with $v_1\not\sim c$. The map $v_i\mapsto v_{i+1}$ for $i\in\{1,\ldots,n-1\}$ and $v_n\mapsto c$ is a monomorphism, so it can be extended to a surjective endomorphism $f$. Consider any vertex $v_0\in f^{-1}[v_1]$: for any $i\in\{1,\ldots,n-1\}$, the pair $v_0v_i$ is mapped to $v_1,v_{i+1}$, and $v_0v_n$ is mapped to $v_1c$. All the images are nonedges, and therefore $\{v_0,\ldots,v_n\}$ is an independent set, contradicting maximality of $\{v_1,\ldots,v_n\}$ as an independent set.
\end{proof}
\fi

\begin{corollary}\label{cor:InfInd}
If $G$ is an MB-homogeneous graph with a nonedge, and $c$ is a co-cone over a finite $A\subset G$, then there exists an infinite independent set of co-cones over $A$.
\end{corollary}
\vspace{-0.6cm}\begin{flushright}$\square$\end{flushright}

\iffalse
\begin{proof}
We show the induction step for an argument to extend a finite independent set of co-cones over $A$. Suppose $C=\{a_1,\ldots,a_n\}$ is a maximal independent set of co-clones over $A$. By Proposition \ref{cor:InfDegCodeg}, there exists $c\notin A\cup C$ with $c\not\sim a_1$. Let $f$ be the function fixing $A$ pointwise and sending $a_i\mapsto a_{i+1}$ for $i\in\{1,\ldots,n-1\}$ and $a_n\mapsto c$. This is a monomorphism, so it extends to a surjective endomorphism $F$. As in Lemma \ref{lem:CoRS} $a_0\in F^{-1}[a_1]$ extends the independent set of co-cones.
\end{proof}
\fi

\begin{theorem}\label{thm:mbsubhe}
Countable MB-homogeneous graphs are HE-homogeneous.
\end{theorem}
\begin{proof}
This is clear for complete and empty graphs, and also for $I_\omega[K_\omega]$. Suppose for a contradiction that $G$ is connected, not complete, MB-homogeneous and not HE-homogeneous. By Lemma \ref{lem:MBalg}, there is a finite subset $H$ such that the set of co-cones of $H$ is finite. This contradicts Corollary \ref{cor:InfInd}.
\end{proof}

%%%%%%%%%%%%%%%%%%%%%%%%%%%%%%%%%%%%%%%%%%%%%%%%%%%%%%%%%%%%
%%%%%%%%%%%SECTION ME=HE%%%%%%%%%%%%%%%%%%%%%%%%%%%%%%%%%%%%
%%%%%%%%%%%%%%%%%%%%%%%%%%%%%%%%%%%%%%%%%%%%%%%%%%%%%%%%%%%%
\section{ME=HE}

Recall that the \emph{independence number} of a graph $G$ is $\alpha(G)=\sup\{n\in\N: I_n\in\age{G}\}$. It was observed in \cite{CameronNesetril:2006} (Proposition 2.1 (c)) that for any HH-homogeneous graph with $\neg$\trg, the value of $\sigma(G)=\sup\{n\in\N: K_{1,n}\in\age{G}\}$ is finite. The following theorem (Theorem 20 from \cite{aranda2019independence}) links these two values for connected HH-homogeneous graphs.

\begin{theorem}\label{thm:indeptrg}
If $G$ is a countably infinite connected HH-homogeneous graph with $\neg$\trg, then $\alpha(G)<2\sigma(G)+\left\lceil\frac{\sigma(G)}{2}\right\rceil-1$. 
\end{theorem}

In particular, connected HH-homogeneous graphs that embed arbitrarily large independent set satisfy \trg.

\begin{theorem}\label{thm:compme}
The complement of an ME-homogeneous graph is MH-homo\-geneous.
\end{theorem}
\begin{proof}
Let $G$ be an ME-homogeneous graph, and suppose that $f\colon A\to B$ is a local surjective monomorphism in $\overline{G}$. Then $f^{-1}\colon B\to A$ is a monomorphism in $G$, which by ME-homogeneity is restriction of some surjective endomorphism $F\colon G\to G$. Now let $\overline F$ be a right inverse of $F$ containing $f$.

We claim that $\overline F$ is an endomorphism of $\overline G$. It is clearly defined on all of $\overline{G}$; now suppose that $v\sim w$ in $\overline G$. This is equivalent ot $v\not\sim w$ in $G$, and therefore, for all $v',w'$ with $F(v')=v,F(w')=w$ we have $v'\not\sim w$. In particular, $\overline{F}(v)\sim\overline{F}(w)$ in $\overline{G}$.
\end{proof}

\begin{corollary}\label{cor:mectrg}
If $G$ is an ME-homogeneous graph with connected complement and $G$ embeds arbitrarily large cliques, then $G$ satisfies \ctrg.
\end{corollary}
\begin{proof}
By Theorems \ref{thm:compme} and \ref{thm:rs}, $\overline{G}$ is an HH-homogeneous graph that embeds arbitrarily large cliques. It follows from Theorem \ref{thm:indeptrg} that $\overline{G}$ satisfies \trg, or, equivalently, $G$ satisfies \ctrg.
\end{proof}

\iffalse
\begin{remark}
Theorems \ref{thm:indeptrg} and \ref{thm:compme} give a bound on the clique number of a ME-homogeneous graph that does not satisfy \ctrg.
\end{remark} 
\fi

\begin{theorem}\label{thm:amon}
A countably infinite graph $G$ is ME-homogeneous iff it is HE-homogeneous.
\end{theorem}
\begin{proof}
HE$\subset$ME is clear.

Suppose first that $G$ is an ME-homogeneous disconnected graph. Then it is of the form $I_\omega[K_\omega]$ or $I_\omega[K_n]$ (Proposition \ref{prop:disconnectedme}). These graphs are HH-homogeneous with \ctrg (which implies \dgr), and therefore HE-homogeneous (Proposition \ref{prop:HHctrg}). 

If the complement of an ME-homogeneous graph $G$ is disconnected, then $\overline{G}$ is an MH-homogeneous graph, and by Theorems \ref{thm:rs} and \ref{thm:2facts}, it is of the form $I_m[K_n]$ with $\max{m,n}=\omega$. We will eliminate some of the possibilities.

We claim that $\overline{G}$ is not of the forms $I_\omega[K_n]$ or $I_n[K_\omega]$ with finite $n\geq 2$ because in this case $G$ would be a connected ME-homogeneous graph embedding nonedges and with finite maximal independent sets, contradicting Proposition \ref{prop:CoconesME} via Remark \ref{rmk:switcheroo}. 

It follows from Lemma \ref{lem:disccliques} and Theorem \ref{thm:iasonhh} that $\overline{G}$ is isomorphic to $I_\omega[K_\omega]$, so $G$ is isomorphic to $\overline{I_\omega[K_\omega]}$. Now $\overline{I_\omega[K_\omega]}$ is an HH-homogeneous graph, and $\hh{\overline{I_\omega[K_\omega]}}$ is the set of finite independent sets, which is clearly downward-closed in $(\age{G},\preceq)$, and therefore $\overline{I_\omega[K_\omega]}$ is HE-homogeneous (Proposition \ref{prop:ConditionsHE}).

Now, if $G$ is connected, ME-homog	eneous, and has connected complement, then from MH-homogeneity it follows that $G$ embeds arbitrarily large cliques. By Corollary \ref{cor:mectrg}, $G$ satisfies \ctrg, so it also satisfies \dgr. By Corollary \ref{cor:dgrHE}, $G$ is HE-homogeneous.
\end{proof}

%%%%%%%%%%%%%%%%%%%%%%%%%%%%%%%%%%%%%%%%%%%%%%%%%%%%%%%%%%%%%%%%%%
%%%%%%%%%%%%%SECTION: ME=HE=MB+DISCONNECTED%%%%%%%%%%%%%%%%%%%%%%%
%%%%%%%%%%%%%%%%%%%%%%%%%%%%%%%%%%%%%%%%%%%%%%%%%%%%%%%%%%%%%%%%%%
\section{$\mathrm{ME=HE=MB\cup\{I_\omega[K_n]:n\geq2\}}$}
\iffalse
\begin{proposition}\label{cor:connMEtrg}
All connected ME-homogeneous graphs satisfy \trg.
\end{proposition}
\begin{proof}
ME-homogeneous structures are MH-homogeneous; in the case of graphs, Theorem \ref{thm:rs} gives HH-homogeneity. 

If $G$ is complete, \trg is trivially true. Otherwise, $G$ contains a nonedge, and by Proposition \ref{prop:infdegcodegr}, $G$ has no finite maximal independent sets, so by Theorem \ref{thm:indeptrg}, $G$ satisfies \trg.
\end{proof}
\fi

\begin{lemma}\label{thm:heconnmb}
If $G$ is a connected HE-homo\-geneous graph, then it is MB-homo\-geneous.
\end{lemma}
\begin{proof}
Any HE-homogeneous graph is also HH-homogeneous, and therefore MH-homogeneous (Theorem \ref{thm:rs}). Since $G$ is connected, Theorem \ref{thm:rs} tells us that $G$ is MM-homogeneous, which is equivalent to satisfying \str (Remark \ref{rmk:strmh}). Now Proposition \ref{prop:NoStr} completes the argument, and $G$ is MB-homogeneous.
\end{proof}

\begin{theorem}\label{thm:mbmeheeq}
In the class of countably infinite connected graphs, HE=ME=MB.
\end{theorem}
\begin{proof}
We know HE=ME (Theorem \ref{thm:amon}) and MB$\subset$HE (Theorem \ref{thm:mbsubhe}) for general countable graphs. Lemma \ref{thm:heconnmb} completes the argument.
\end{proof}

\begin{proposition}\label{prop:disconnectedme}
The only disconnected ME-homogeneous graphs are $I_\omega[K_\omega]$ and $I_\omega[K_n]$ for $n\in\N$.
\end{proposition}
\begin{proof}
Conditions \ref{Cond5} and \ref{Cond6a} from Proposition \ref{prop:ConditionsME} are easy to verify for each graph in the statement. On the other hand, a graph of the form $I_n[K_\omega]$  (finite $n>1$) cannot be ME-homogeneous, because the image of any endomorphism extending the monomorphism that maps a nonedge to an edge has fewer than $n$ connected connected components.
\end{proof}

\begin{theorem}\label{thm:discmehemb}
In the class of countably infinite graphs, $$\mathrm{ME=HE=MB\cup\{I_\omega[K_n]:n\geq2\}}.$$
\end{theorem}
\begin{proof}
Clearly MB$\subseteq$ME. The rest follows from Theorem \ref{thm:amon}, Proposition \ref{prop:disconnectedme}, and Proposition \ref{cor:InfDegCodeg}.
\end{proof}

\iffalse
\begin{corollary}
Every connected HE-homogeneous graph with connected complement is MB-homogeneous and bimorphism-equivalent to the Rado graph.
\end{corollary}
\begin{proof}
The proof of Theorem \ref{thm:amon} says in particular that every connected HE-homogeneous graph satisfies \ctrg. We already knew that these graphs satisfy \trg, and the result follows.
\end{proof}
\fi

%%%%%%%%%%%%%%%%%%%%%%%%%%%%%%%%%%%%%%%%%%%%%%%%%%%%%%%%%%%%%%%%%%
%%%%%%%%%%%%%SECTION: PROPER INCLUSIONS%%%%%%%%%%%%%%%%%%%%%%%%%%%
%%%%%%%%%%%%%%%%%%%%%%%%%%%%%%%%%%%%%%%%%%%%%%%%%%%%%%%%%%%%%%%%%%
\section{Proper inclusions}
In this section we present examples that separate the morphism-extension classes. 

So far, we know for countable connected graphs:
\begin{enumerate}
\item{HM=HI=HB=HA=MI=MA (Theorem \ref{thm:bottom})}
\item{II=IA (Lemma 1.1 in \cite{LockettTruss:2014})}
\item{ME=MB=HE (Theorem \ref{thm:mbmeheeq})}
\item{MM=MH=HH (Theorem \ref{thm:rs})}
\end{enumerate}

And for general countable graphs:
\begin{enumerate}
\item{HM=HI=HB=HA (Theorem \ref{thm:bottom})}
\item{MI=MA (Lemma 1.1 in \cite{LockettTruss:2014})}
\item{II=IA (Lemma 1.1 in \cite{LockettTruss:2014})}
\item{ME=HE (Theorem \ref{thm:amon}}
\item{MH=HH (Theorem \ref{thm:rs})}
\end{enumerate}

Moreover, we know that the only examples separating MM from HH=MH and ME=HE from MB are disconnected graphs (Theorems \ref{thm:rs} and \ref{thm:discmehemb}). Now we will give examples witnessing all other proper containments.

\begin{example}[ME=HE$\subset$MH=HH; MB$\subset$MM]
The Rusinov-Schweitzer examples are HH-homogeneous and connected, therefore in MM and MH=HH (Theorem \ref{thm:rs}), but they are not in the Aranda-Hartman catalogue (Theorem \ref{thm:ah}), because they have finite independence number at least 2. This proves MB$\subset$ MM.

The same example proves ME$\subset$HH. If $n\geq 3$, then there is exactly one independent set of size $n$ in $RS(n)$, say $Z$. Consider the monomorphism mapping $Z$ to a clique of size $n$. This is not restriction of any epimorphism, because the preimage of $Z$ under any epimorphism of $RS(n)$ is $Z$ itself.
\end{example}

\begin{example}[ME=HE$\subset$IE; MH=HH$\subset$IM]
The universal homogeneous \linebreak$K_n$-free graph is in IE and IM, but not in ME or HH (see Observation \ref{obs:hn}).
\end{example}

\begin{example}[IB$\subset$IE; IM$\subset$IH]
Let $H_3=(V,E)$ denote the universal homogeneous triangle-free graph. Fix a nonedge $u,v$ and a vertex $w\in N(u)\cap N(v)$. Define $H_3'$ as $H_3'=(V,E')$, where $E'$ is the subset of $E$ obtained as follows:
\begin{enumerate}
\item{Partitioning $(N(u)\cap N(v))\setminus\{w\}$ into two infinite subsets $C_u$ and $C_w$.}
\item{$E'=E\setminus(\{xv:x\in C_u\}\cup\{xu:x\in C_v\})$}
\end{enumerate}
It is clear that $H_3'$ satisfies \ctrg, and therefore \dgr. Next, we prove that $H_3'$ is IH-homogeneous, which, by Proposition \ref{prop:ihdgr}, suffices to prove IE-homogeneity.

Suppose that $f:A\to B$ is an isomorphism between finite substructures of $H_3'$. Let $c\notin A$.
\begin{enumerate}
\item{If $\{u,v,w\}\not\subset B$, then write $X=\{u,v,w\}\setminus B$. We can find $c'$ with $c'\sim f(s)\leftrightarrow c\sim s$ with the additional requirement $c'\not\sim x$ for all $x\in X$.}
\item{If $\{u,v,w\}\subset B$, then map $c\to w$ if $c\sim f^{-1}(v)\wedge c\sim f^{-1}(w)$. In all other cases we can find $c'$ as above.}
\end{enumerate}
Finally, we prove that $H_3'$ is not IM-homogeneous (and therefore not IB-homo\-geneous). To do this, consider the isomorphism $f\colon c\mapsto u, d\mapsto v$, where $c,d$ is a nonedge and $\{c,d\}\neq\{u,v\}$. This isomorphism cannot be extended to an injective endomorphism because $c$ and $d$ have infinitely many common neighbours, but $u$ and $v$ have only one.
\end{example}

\begin{example}[IE$\subset$IH]
Let $\mathcal{R}=(V,E)$ be the Rado graph, $w\notin V$, and $\mathcal{R}'=(V\cup\{w\},E\cup\{wx:x\in V\})$. 

Now $R'$ is not IE-homogeneous because for any vertex $v\neq w$ there exists $v'$ such that $v,v'$ is a nonedge. It follows that mapping $w$ to any $v\in V$ cannot be extended to a surjective mapping. $\mathcal{R}'$ satisfies property \trg, and so it is MH-homogeneous and therefore IH-homogeneous.
\end{example}

The information we have gathered suffices to draw the posets of morphism-extension classes of graphs and connected graphs.
\begin{figure}[h!]
\centering
\includegraphics[scale=0.8]{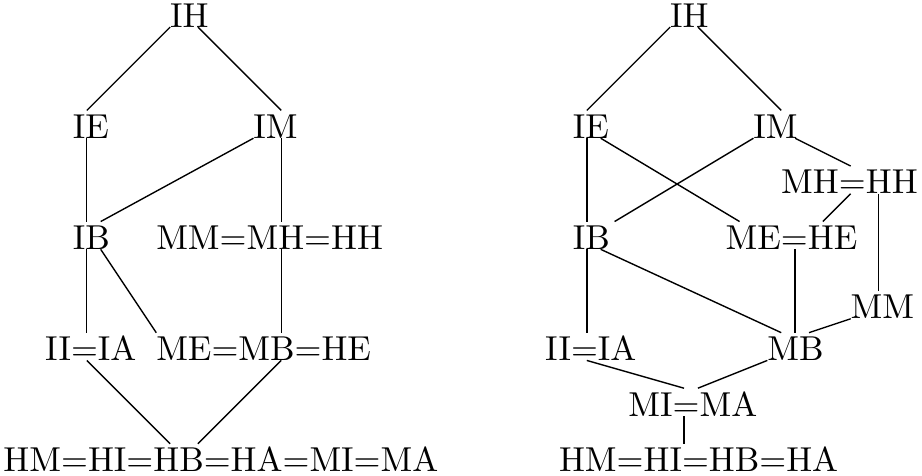}
\caption{Morphism-extension classes of countable graphs (right), and countable connected graphs (left), ordered by $\subseteq$.}
\end{figure}

\section{Ackowledgements}
I thank Thomas D.H. Coleman for the fruitful correspondence of early 2018. 

Research funded by the ERC under the European Union's Horizon 2020 Research and Innovation Programme (grant agreement No. 681988, CSP-Infinity).

%\section*{References}
\bibliography{Morph}
\end{document}